\documentclass{amsart}
\usepackage{amssymb}
\usepackage{amsfonts}
\usepackage{amsmath}
\usepackage{psfig}
\usepackage{psfrag}
\usepackage{graphicx}

\newtheorem{theorem}{Theorem}[section]
\newtheorem{thm}[theorem]{Theorem}

\newtheorem{pro}[theorem]{Proposition}
\newtheorem{defi}[theorem]{Definition}

\newtheorem{rmk}[theorem]{Remark}

\numberwithin{equation}{section}

\begin{document}

\title{Decidability of plane edge coloring with three colors}


\author{Hung-Hsun Chen}
\address{Department of Applied Mathematics, National Chiao Tung University, Hsinchu 300, Taiwan}
\email{hhchen.am00g@nctu.edu.tw}

\author{Wen-Guei Hu}
\address{Department of Applied Mathematics, National Chiao Tung University, Hsinchu 300, Taiwan}
\email{wghu@mail.nctu.edu.tw}

\author{De-Jan Lai}
\address{Department of Applied Mathematics, National Chiao Tung University, Hsinchu 300, Taiwan}
\email{werre216asfe87dirk@gmail.com}

\author{Song-Sun Lin$^{\star}$}
\address{Department of Applied Mathematics, National Chiao Tung University, Hsinchu 300, Taiwan}
\email{sslin@math.nctu.edu.tw}
\thanks{$^{\star}$The author would like to thank the National Science Council, R.O.C. (Contract No. NSC 98-2115-M-009-008) for partially supporting this research.}

%

\begin{abstract}
This investigation studies the decidability problem of plane edge coloring with three symbols. In the edge coloring (or Wang tiles) of a plane, unit squares with
colored edges that have one of $p$ colors are arranged side by side such that the touching edges of the adjacent tiles have the same colors. Given a basic set $B$ of Wang tiles, the decision problem is to find an algorithm to determine whether or not $\Sigma(B)\neq\emptyset$, where $\Sigma(B)$ is the set of all global patterns on $\mathbb{Z}^{2}$ that can be constructed from the Wang tiles in $B$.

When $p\geq 5$, the problem is known to be undecidable. When $p=2$, the problem is decidable. This study proves that when $p=3$, the problem is also decidable. $\mathcal{P}(B)$ is the set of all periodic patterns on $\mathbb{Z}^{2}$ that can be generated by the tiles in $B$. If $\mathcal{P}(B)\neq\emptyset$, then $B$ has a subset $B'$ of minimal cycle generators such that $\mathcal{P}(B')\neq\emptyset$ and $\mathcal{P}(B'')=\emptyset$ for $B''\subsetneqq B'$. This study demonstrates that the set $\mathcal{C}(3)$ of all minimal cycle generators  contains $787,605$ members that can be classified into $2,906$ equivalence classes. $\mathcal{N}(3)$ is the set of all maximal non-cycle generators : if $B\in \mathcal{N}(3)$, then $\mathcal{P}(B)=\emptyset$ and $\mathcal{P}(\tilde{B})\neq\emptyset$ for $\tilde{B}\supsetneqq B$. The problem is shown to be decidable by proving that $B\in \mathcal{N}(3)$ implies $\Sigma(B)=\emptyset$. Consequently, $\Sigma(B)\neq\emptyset$ if and only if $\mathcal{P}(B)\neq\emptyset$.
\end{abstract}

\maketitle

\section{ Introduction}

Decidability problems have been studied for many years. See, for example, the recent review article of Goodman-Strauss \cite{5-1}. One of the most active areas of research into the
decidability problem is that of the plane tiling \cite{6}. This study focuses on the
decidability problem concerning plane edge coloring with three symbols.

The coloring of $\mathbb{Z}^2$ using unit squares has a long history \cite{6}. In 1961, Wang \cite{10}
started to study the square tiling of a plane to prove theorems by pattern recognition. Unit
 squares with colored edges are arranged side by side so that the touching edges of the adjacent tiles have the same color; the tiles cannot
 be rotated or reflected. Today, such tiles are called Wang tiles or Wang dominos \cite{4,6}.

The $2 \times 2$ unit square is denoted by $\mathbb{Z}_{2 \times 2}$. The set of $p$ colors is $\{0,1,\cdots,p-1\}$. Therefore, the total set of Wang tiles is denoted by
$\Sigma_{2 \times 2}(p) \equiv \{0,1,\cdots,p-1\}^{\mathbb{Z}_{2 \times 2}}$. A set $B$ of Wang tiles is called a basic set (of Wang tiles). Let $\Sigma(B)$ and $\mathcal{P}(B)$ be the sets of all global patterns and periodic patterns on $\mathbb{Z}^2$, respectively, that can be constructed from the Wang tiles in $B$.

The decision problem concerning tiling with of Wang tiles concerns the existence of an algorithm that can determine whether or not
\begin{equation}  \label{eqn:1.1}
\Sigma(B) \neq \emptyset
\end{equation}
for any finite set $B$ of Wang tiles.

Clearly, $\mathcal{P}(B) \subseteq \Sigma(B)$, meaning that if $\mathcal{P}(B) \neq \emptyset$,
then $\Sigma(B) \neq \emptyset$. In
\cite{10}, Wang conjectured that any set of tiles that can tile a plane can tile the plane periodically:
\begin{equation}  \label{eqn:1.2}
\text{ if } \Sigma(B) \neq \emptyset, \text{ then } \mathcal{P}(B) \neq \emptyset.
\end{equation}

If (\ref{eqn:1.2}) holds, then the decision problem that is spcified by (\ref{eqn:1.1}) is reduced to the much easier problem of determining whether or not
\begin{equation}  \label{eqn:1.3}
\mathcal{P} (B) \neq \emptyset.
\end{equation}

However, in 1966, Berger \cite{4} proved that Wang's conjecture was wrong and the decision problem concerning Wang's tiling is undecidable.
He presented a set $B$ of $20426$ Wang tiles that could only tile the plane aperiodically:
\begin{equation}  \label{eqn:1.4}
\begin{array}{ccc}
\Sigma(B)\neq \emptyset & \text{ and } & \mathcal{P}(B)= \emptyset.
\end{array}
\end{equation}
Later, he reduced the number of tiles to $104$. Thereafter, smaller basic sets were found
by Knuth, L\"{a}uchli, Robinson, Penrose, Ammann,  Culik and Kari. Currently, the smallest number
of tiles that can tile a plane aperiodically is $13$, with five colors: (\ref{eqn:1.4}) holds and then
(\ref{eqn:1.2}) fails for $p=5$ \cite{5}.

Recently, Hu and Lin \cite{11} showed that Wang's conjecture (\ref{eqn:1.2}) holds if $p=2$: any
set of Wang tiles with two colors that can tile a plane can tile the plane periodically.

In that study, they showed that statement (\ref{eqn:1.2})can be approached by studying how periodic
patterns can be generated from a given basic set. First, $B$ is called a cycle generator if $\mathcal{P}(B)\neq
\emptyset$; otherwise, $B$ is called a non-cycle generator. Moreover,
$B\subset\Sigma_{2\times 2}(p)$ is
called a minimal cycle generator (MCG) if $B$ is a cycle generator and
$\mathcal{P}(B')=\emptyset$ whenever $B'\subsetneq B$; $B\subset\Sigma_{2\times 2}(p)$ is called a maximal non-cycle generator (MNCG) if $B$ is a non-cycle generator
 and $\mathcal{P}(B'')\neq\emptyset$ for any
$B''\supsetneqq B$.

Given $p\geq 2$, denote
the set of all minimal cycle generators by $\mathcal{C}(p)$ and
the set of maximal non-cycle generators by $\mathcal{N}(p)$.
Clearly,
\begin{equation}  \label{eqn:1.5}
\mathcal{C}(p)\cap\mathcal{N}(p)=\emptyset.
\end{equation}
Statement (\ref{eqn:1.2}) follows for $p\geq 2$ if

\begin{equation}  \label{eqn:1.6}
\Sigma(B) \neq \emptyset \text{ for any } B \in \mathcal{N}(p)
\end{equation}
can be shown. Indeed, in \cite{11}, it is shown that $\mathcal{C}(2)$ has 38 members ;$\mathcal{N}(2)$ has nine members, and (\ref{eqn:1.6})
holds for $p=2$. This paper studies the case of $p = 3$. Now, $\mathcal{C}(3)$ and $\mathcal{N}(3)$ have close to a million members and cannot be handled manually. After the symmetry group $D_4$ of $\mathbb{Z}_{2 \times 2}$ and the permutation group $S_p$ of colors of horizontal and vertical edges, respectively, are applied, $\mathcal{C}(3)$ still contains thousands of equivalent classes. Hence, computer programs are utilized to determine $\mathcal{C}(3)$ and $\mathcal{N}(3)$ and finally (\ref{eqn:1.6}) is shown to hold for $p = 3$. Therefore, the problem (\ref{eqn:1.1}) is decidable for $p=3$.

For $p = 4$, $\mathcal{C}(4)$ is enormous. Therefore, the arguments and the computer program need to be much efficient to handle this situation.

Corner coloring with $p = 3$ can be treated similarly. The result will be presented elsewhere.

The rest of paper is arranged as follows. Section 2 introduces the ordering matrix of all 81 local patterns and
classifies them into three groups. The recurrence formula for patterns on $\Sigma_{m \times n}$ are derived. It
is important in proving (\ref{eqn:1.6}) - that the maximum non-cycle generators cannot generate global patterns.
Section 3 will introduce the procedure for determining the sets $\mathcal{C}(3)$ and $\mathcal{N}(3)$. The main result is proven using a computer.

\section{Preliminary }
\setcounter{equation}{0}
This section introduces all necessary elements for proving (\ref{eqn:1.6}). First, let $\Sigma_{m \times n}(B)$ be the set of all local patterns on $\mathbb{Z}_{m \times n}$ that can be  generated by $B$.
Clearly,

\begin{equation}  \label{eqn:2.1-0}
\text{if } \Sigma_{m \times n}( B ) = \emptyset \text{ for some } m, n \geq 2, \text{ then }
\Sigma( B ) = \emptyset.
\end{equation}

\subsection{ Symmetries }
The symmetry of the unit square $\mathbb{Z}_{2 \times 2}$ is introduced. The symmetry group of the rectangle $\mathbb{Z}_{2 \times 2}$ is $D_4$, which is the dihedral group of order eight. The group $D_4$ is generated by rotation $\rho$ through $\frac{\pi}{2}$ and reflection $m$ about the $y$-axis. Denote the elements of $D_4$ by $D_4 = \{ I, \rho, \rho^2, \rho^3, m, m\rho, m\rho^2, m\rho^3 \}$.

\begin{equation*}
\psfrag{a}{\tiny{$m \rho$}}
\psfrag{b}{\tiny{$m \rho^2$}}
\psfrag{c}{\tiny{$m \rho^3$}}
\psfrag{d}{\tiny{$m$}}
\psfrag{e}{\tiny{$\rho$}}
\includegraphics[scale=0.6]{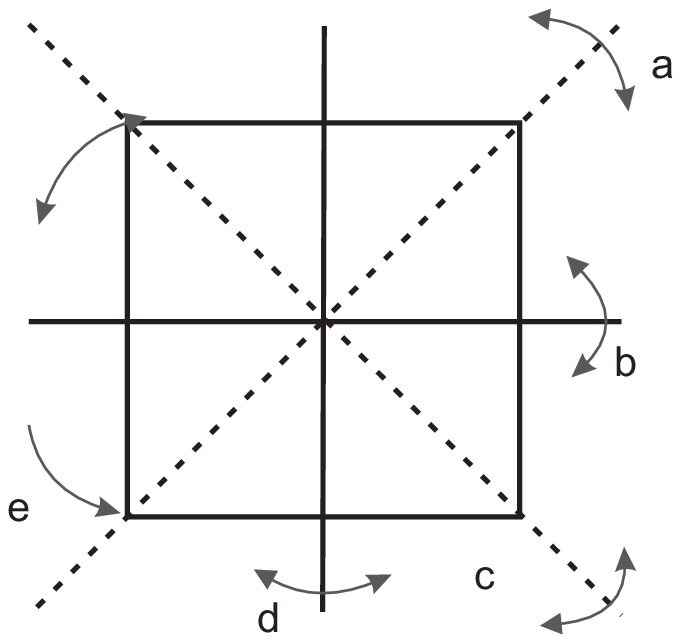}
\end{equation*}

\begin{equation*}
\text{Figure 2.1}
\end{equation*}
Therefore, given a basic set $B \subset \Sigma_{2 \times 2}(p)$ and any element $\tau \in D_4$, another basic set $(B)_\tau$ can be obtained by transforming the local patterns in $B$ by $\tau$.

Additionally, consider the permutation group $S_p$ on $\{0,1,\cdots, p-1\}$. If $\eta \in S_p$ and $\eta(0) = i_0, \eta(1) = i_1, \cdots, \eta(p-1) = i_{p-1}$, we write
$$
\eta = \left( \begin{array}{cccc}
              0 & 1 & \cdots & p-1 \\ i_0 & i_1 & \cdots & i_{p-1}
              \end{array}  \right).
$$
For $\eta \in S_p$ and $B \in \Sigma_{2 \times 2}(p)$, another basic set $(B)_{\eta}$ can be obtained.

In edge coloring, the permutations of colors in the horizontal and vertical directions are mutually independent. Denote the permutations of colors in the horizontal and vertical edges by $\eta_h$ $\in$ $S_p$ and $\eta_v$ $\in$ $S_p$, respectively. Now, for any $B$ $\subset$ $\Sigma_{2 \times 2}(p)$, define the equivalence class $[B]$ of $B$ by

\begin{equation} \label{eqn:2.2}
[B] = \{ B' \subset \Sigma_{ 2 \times 2}(p) : B' =
(((B)_{\tau})_{\eta_h})_{\eta_v}, \tau \in D_4 \mbox{ and } \eta_h, \eta_v \in ~ S_p \}.
\end{equation}

In \cite{11}, whether or not $\Sigma(B)=\emptyset$ and $\mathcal{P}(B)=\emptyset$ is shown to be independent of the
choice of elements in $[B]$. Indeed, for any
$B'\in[B]$,

\begin{equation*}
\begin{array}{ccc}
\Sigma(B')\neq \emptyset \hspace{0.2cm} (\text{or }\mathcal{P}(B')\neq \emptyset)& \text{if and only if} & \Sigma(B)\neq \emptyset\hspace{0.2cm} (\text{or }\mathcal{P}(B)\neq \emptyset).
\end{array}
\end{equation*}
Moreover, for $B'\in[B]$, $B'$ is an MCG (MNCG) if and only if $B$ is an MCG (MNCG). Therefore, groups $D_4$ and $S_{3}$ can be used efficiently to reduce the number of cases $B\subset\Sigma_{2\times 2}(3)$ that must be considered, greatly reducing the computation time.

\subsection{ Ordering Matrix }

Now, the case $p = 3$ is considered. The vertical ordering matrix $\mathbf{Y}_{2 \times 2} = \left[ y_{i,j} \right]_{9\times 9}$ of all local patterns in $\Sigma_{2 \times 2}(p)$ is given by

\begin{eqnarray}
\mathbf{Y}_{2\times 2} &=&
\begin{array}{c}
\includegraphics[scale=0.4]{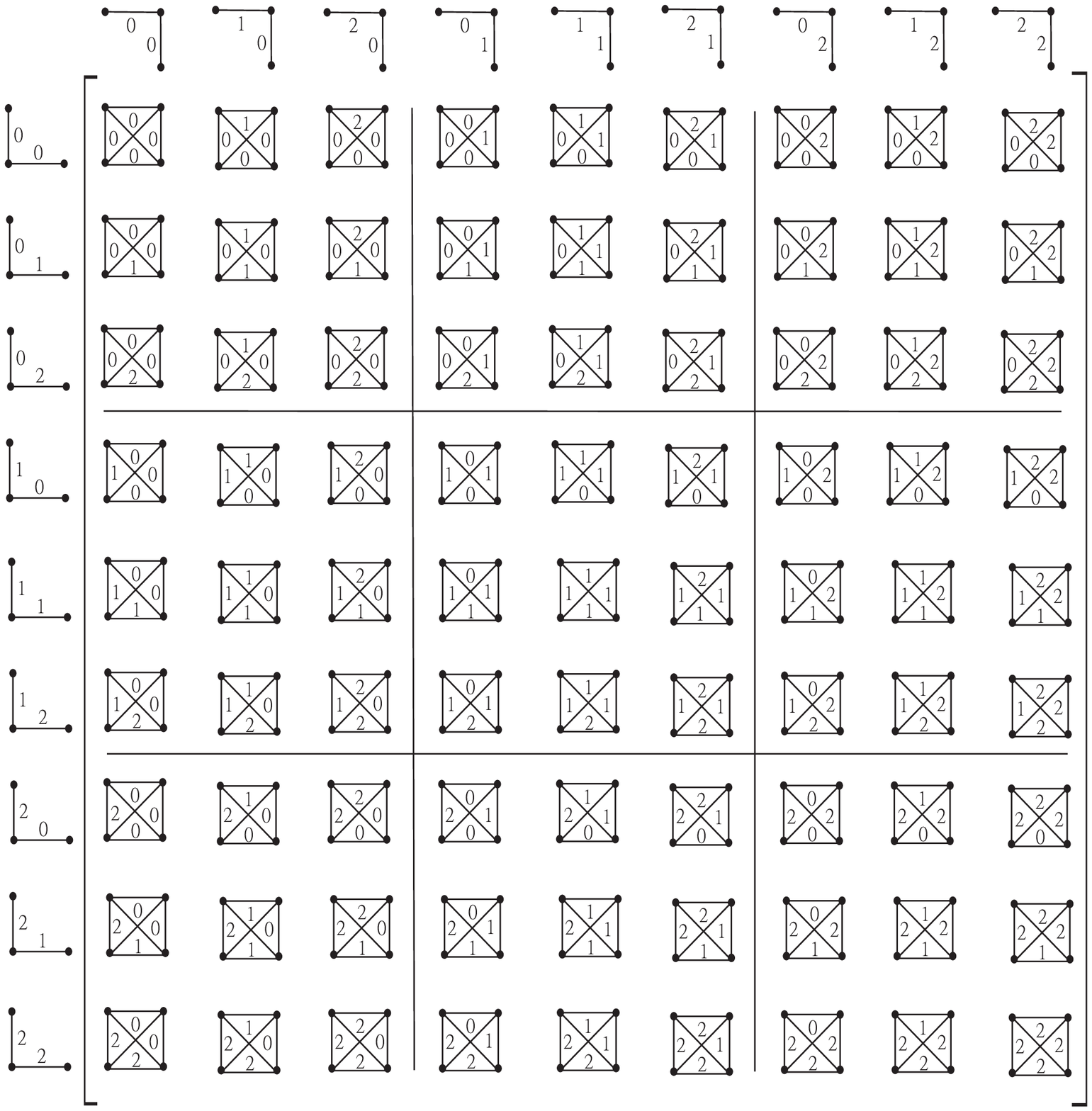}
\end{array} \\
&=& \left[
\begin{array}{ccc}
\mathbf{Y}_{2;1} & \mathbf{Y}_{2;2} & \mathbf{Y}_{2;3} \\
\mathbf{Y}_{2;4} & \mathbf{Y}_{2;5} & \mathbf{Y}_{2;6} \\
\mathbf{Y}_{2;7} & \mathbf{Y}_{2;8} & \mathbf{Y}_{2;9} \\
\end{array}
\right]_{3 \times 3}
\end{eqnarray}

The recurrence relation of $\mathbf{Y}_{m+1}$ is easily obtained as follows.
Denote by
\begin{eqnarray} \label{eqn:2.5}
\mathbf{Y}_{2} = \underset{i=1}{\overset{9}{\sum}}\hspace{0.1cm} \mathbf{Y}_{2;i}
\end{eqnarray}
and
\begin{eqnarray} \label{eqn:2.6}
\mathbf{Y}_{2;i} = \left[ \begin{array}{c} y_{2;i;p,q} \end{array} \right]_{3 \times 3},
\end{eqnarray}
where
\begin{equation*}
y_{2;i;p,q} =
\begin{array}{c}
\psfrag{a}{\tiny{$p-1$}}
\psfrag{b}{\tiny{$q-1$}}
\psfrag{c}{\tiny{$\alpha_{1}$}}
\psfrag{d}{\tiny{$\alpha_{2}$}}
\includegraphics[scale=0.4]{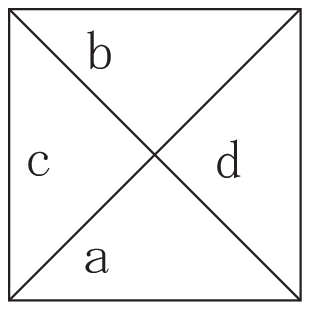}
\end{array}
\end{equation*}
and $i = 1 + \alpha_1 \cdot 3^1 + \alpha_2 \cdot 3^0$, $\alpha_i \in \{ 0, 1, 2 \}$.
For $m \geq 2$, denote by
\begin{eqnarray} \label{eqn:2.7}
\mathbf{Y}_{m+1} = \underset{i=1}{\overset{9}{\sum}}\hspace{0.1cm}  \mathbf{Y}_{m+1;i}
\end{eqnarray}
and
\begin{eqnarray} \label{eqn:2.8}
\mathbf{Y}_{m+1;i}= \left[ \begin{array}{c} y_{m+1;i;p,q} \end{array} \right]_{3^{m} \times 3^{m}},
\end{eqnarray}
where $y_{m+1;i;p,q}$ is the set of all patterns of the form
\begin{equation*}
\begin{array}{rcl}
\psfrag{a}{$\bullet$}
\psfrag{b}{\tiny{$\alpha$}}
\psfrag{c}{\tiny{$\beta$}}
\psfrag{e}{$\cdots$}
\psfrag{m}{\tiny{$m$ tiles}}
\psfrag{f}{\tiny{$p_{1}$}}
\psfrag{g}{\tiny{$p_{2}$}}
\psfrag{h}{\tiny{$p_{m}$}}
\psfrag{k}{\tiny{$q_{1}$}}
\psfrag{l}{\tiny{$q_{2}$}}
\psfrag{n}{\tiny{$q_{m}$}}
\includegraphics[scale=0.45]{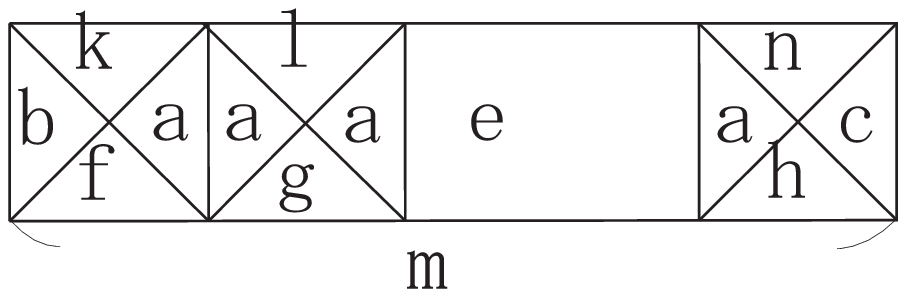}
\end{array}
\end{equation*}
where $\alpha$ ,$\beta$, $p_{k}$ and $q_{k}\in\{0,1,2\}$, $1\leq k\leq m$, such that
\begin{equation} \label{eqn:2.9}
\left\{
\begin{array}{l}
i = 1 + \alpha\cdot3^{1} + \beta\cdot  3^{0} \\
p= 1 + \underset{k=1}{\overset{m}{\sum}}p_{k}3^{m-k} \\
q= 1 + \underset{k=1}{\overset{m}{\sum}}q_{k}3^{m-k} \\
\end{array}
\right.
\end{equation}
and  $\bullet \in \{0, 1, 2 \}$. Therefore, when $i =1,4,7$,
\begin{eqnarray*}
\scriptsize{
\mathbf{Y}_{m+1;i}
= \left[
\begin{array}{ccc}
\underset{j=1}{\overset{3}{\sum}} y_{2;j+i-1;1,1} \mathbf{Y}_{m;3j-2} &
\underset{j=1}{\overset{3}{\sum}} y_{2;j+i-1;1,2} \mathbf{Y}_{m;3j-2} &
\underset{j=1}{\overset{3}{\sum}} y_{2;j+i-1;1,3} \mathbf{Y}_{m;3j-2} \\
\underset{j=1}{\overset{3}{\sum}} y_{2;j+i-1;2,1} \mathbf{Y}_{m;3j-2} &
\underset{j=1}{\overset{3}{\sum}} y_{2;j+i-1;2,2} \mathbf{Y}_{m;3j-2} &
\underset{j=1}{\overset{3}{\sum}} y_{2;j+i-1;2,3} \mathbf{Y}_{m;3j-2} \\
\underset{j=1}{\overset{3}{\sum}} y_{2;j+i-1;3,1} \mathbf{Y}_{m;3j-2} &
\underset{j=1}{\overset{3}{\sum}} y_{2;j+i-1;3,2} \mathbf{Y}_{m;3j-2} &
\underset{j=1}{\overset{3}{\sum}} y_{2;j+i-1;3,3} \mathbf{Y}_{m;3j-2} \\
\end{array}
\right]_{3^{m} \times 3^{m}}
};
\end{eqnarray*}
when $i =2,5,8$,
\begin{eqnarray*}
\scriptsize{
\mathbf{Y}_{m+1;i}
= \left[
\begin{array}{ccc}
\underset{j=1}{\overset{3}{\sum}} y_{2;j+i-2;1,1} \mathbf{Y}_{m;3j-1} &
\underset{j=1}{\overset{3}{\sum}} y_{2;j+i-2;1,2} \mathbf{Y}_{m;3j-1} &
\underset{j=1}{\overset{3}{\sum}} y_{2;j+i-2;1,3} \mathbf{Y}_{m;3j-1} \\
\underset{j=1}{\overset{3}{\sum}} y_{2;j+i-2;2,1} \mathbf{Y}_{m;3j-1} &
\underset{j=1}{\overset{3}{\sum}} y_{2;j+i-2;2,2} \mathbf{Y}_{m;3j-1} &
\underset{j=1}{\overset{3}{\sum}} y_{2;j+i-2;2,3} \mathbf{Y}_{m;3j-1} \\
\underset{j=1}{\overset{3}{\sum}} y_{2;j+i-2;3,1} \mathbf{Y}_{m;3j-1} &
\underset{j=1}{\overset{3}{\sum}} y_{2;j+i-2;3,2} \mathbf{Y}_{m;3j-1} &
\underset{j=1}{\overset{3}{\sum}} y_{2;j+i-2;3,3} \mathbf{Y}_{m;3j-1} \\
\end{array}
\right]_{3^{m} \times 3^{m}}
};
\end{eqnarray*}
when $i =3,6,9$,
\begin{eqnarray*}
\scriptsize{
\mathbf{Y}_{m+1;i}
= \left[
\begin{array}{ccc}
\underset{j=1}{\overset{3}{\sum}} y_{2;j+i-3;1,1} \mathbf{Y}_{m;3j} &
\underset{j=1}{\overset{3}{\sum}} y_{2;j+i-3;1,2} \mathbf{Y}_{m;3j} &
\underset{j=1}{\overset{3}{\sum}} y_{2;j+i-3;1,3} \mathbf{Y}_{m;3j} \\
\underset{j=1}{\overset{3}{\sum}} y_{2;j+i-3;2,1} \mathbf{Y}_{m;3j} &
\underset{j=1}{\overset{3}{\sum}} y_{2;j+i-3;2,2} \mathbf{Y}_{m;3j} &
\underset{j=1}{\overset{3}{\sum}} y_{2;j+i-3;2,3} \mathbf{Y}_{m;3j} \\
\underset{j=1}{\overset{3}{\sum}} y_{2;j+i-3;3,1} \mathbf{Y}_{m;3j} &
\underset{j=1}{\overset{3}{\sum}} y_{2;j+i-3;3,2} \mathbf{Y}_{m;3j} &
\underset{j=1}{\overset{3}{\sum}} y_{2;j+i-3;3,3} \mathbf{Y}_{m;3j} \\
\end{array}
\right]_{3^{m} \times 3^{m}}
}.
\end{eqnarray*}

Given $B\subset\Sigma_{2\times 2}(3)$, the
associated vertical transition matrix $\mathbf{V}_{2\times 2}(B)$ is defined by
$\mathbf{V}_{2\times 2}(B)=\left[ v_{i,j} \right]$, where $v_{i,j}=1$ if and
only if $y_{i,j}\in B$.

The recurrence formula for a higher-order vertical transition matrix can be obtained as follows. Denote by
\begin{eqnarray*}
\mathbf{V}_{2}  = \underset{i=1}{\overset{9}{\sum}} \mathbf{V}_{2;i},
\end{eqnarray*}
with
\begin{eqnarray*}
\mathbf{V}_{2;i} &=& \left[ \begin{array}{c} v_{2;i;p,q} \end{array} \right]_{3 \times 3}.
\end{eqnarray*}
For $m \geq 2$, denote by
\begin{eqnarray*}
\mathbf{V}_{m+1} = \underset{i=1}{\overset{9}{\sum}} \mathbf{V}_{m+1;i}.
\end{eqnarray*}
Now, for $i =1,4,7$,
\begin{eqnarray*}
\scriptsize{
\mathbf{V}_{m+1;i}
= \left[
\begin{array}{ccc}
\underset{j=1}{\overset{3}{\sum}} v_{2;j+i-1;1,1} \mathbf{V}_{m;3j-2} &
\underset{j=1}{\overset{3}{\sum}} v_{2;j+i-1;1,2} \mathbf{V}_{m;3j-2} &
\underset{j=1}{\overset{3}{\sum}} v_{2;j+i-1;1,3} \mathbf{V}_{m;3j-2} \\
\underset{j=1}{\overset{3}{\sum}} v_{2;j+i-1;2,1} \mathbf{V}_{m;3j-2} &
\underset{j=1}{\overset{3}{\sum}} v_{2;j+i-1;2,2} \mathbf{V}_{m;3j-2} &
\underset{j=1}{\overset{3}{\sum}} v_{2;j+i-1;2,3} \mathbf{V}_{m;3j-2} \\
\underset{j=1}{\overset{3}{\sum}} v_{2;j+i-1;3,1} \mathbf{V}_{m;3j-2} &
\underset{j=1}{\overset{3}{\sum}} v_{2;j+i-1;3,2} \mathbf{V}_{m;3j-2} &
\underset{j=1}{\overset{3}{\sum}} v_{2;j+i-1;3,3} \mathbf{V}_{m;3j-2} \\
\end{array}
\right]_{3^{m} \times 3^{m}}
};
\end{eqnarray*}
for $i =2,5,8$,
\begin{eqnarray*}
\scriptsize{
\mathbf{V}_{m+1;i}
= \left[
\begin{array}{ccc}
\underset{j=1}{\overset{3}{\sum}} v_{2;j+i-2;1,1} \mathbf{V}_{m;3j-1} &
\underset{j=1}{\overset{3}{\sum}} v_{2;j+i-2;1,2} \mathbf{V}_{m;3j-1} &
\underset{j=1}{\overset{3}{\sum}} v_{2;j+i-2;1,3} \mathbf{V}_{m;3j-1} \\
\underset{j=1}{\overset{3}{\sum}} v_{2;j+i-2;2,1} \mathbf{V}_{m;3j-1} &
\underset{j=1}{\overset{3}{\sum}} v_{2;j+i-2;2,2} \mathbf{V}_{m;3j-1} &
\underset{j=1}{\overset{3}{\sum}} v_{2;j+i-2;2,3} \mathbf{V}_{m;3j-1} \\
\underset{j=1}{\overset{3}{\sum}} v_{2;j+i-2;3,1} \mathbf{V}_{m;3j-1} &
\underset{j=1}{\overset{3}{\sum}} v_{2;j+i-2;3,2} \mathbf{V}_{m;3j-1} &
\underset{j=1}{\overset{3}{\sum}} v_{2;j+i-2;3,3} \mathbf{V}_{m;3j-1} \\
\end{array}
\right]_{3^{m} \times 3^{m}}
};
\end{eqnarray*}
for $i =3,6,9$,
\begin{eqnarray*}
\scriptsize{
\mathbf{V}_{m+1;i}
= \left[
\begin{array}{ccc}
\underset{j=1}{\overset{3}{\sum}} v_{2;j+i-3;1,1} \mathbf{V}_{m;3j} &
\underset{j=1}{\overset{3}{\sum}} v_{2;j+i-3;1,2} \mathbf{V}_{m;3j} &
\underset{j=1}{\overset{3}{\sum}} v_{2;j+i-3;1,3} \mathbf{V}_{m;3j} \\
\underset{j=1}{\overset{3}{\sum}} v_{2;j+i-3;2,1} \mathbf{V}_{m;3j} &
\underset{j=1}{\overset{3}{\sum}} v_{2;j+i-3;2,2} \mathbf{V}_{m;3j} &
\underset{j=1}{\overset{3}{\sum}} v_{2;j+i-3;2,3} \mathbf{V}_{m;3j} \\
\underset{j=1}{\overset{3}{\sum}} v_{2;j+i-3;3,1} \mathbf{V}_{m;3j} &
\underset{j=1}{\overset{3}{\sum}} v_{2;j+i-3;3,2} \mathbf{V}_{m;3j} &
\underset{j=1}{\overset{3}{\sum}} v_{2;j+i-3;3,3} \mathbf{V}_{m;3j} \\
\end{array}
\right]_{3^{m} \times 3^{m}}
}.
\end{eqnarray*}
Therefore, as in \cite{11}, it can be proven that
\begin{eqnarray} \label{eqn:2.10}
\left|\Sigma_{(m+1) \times n}(B)\right| = \left| \mathbf{V}^{n-1}_{m+1} \right|.
\end{eqnarray}

\subsection{Periodic Patterns}

This subsection studies periodic patterns in detail.

For $m,n\geq1$, a global pattern $u = (\alpha_{i, j})_{i, j \in \mathbb{Z}}$ on $\mathbb{Z}^2$ is called $(m,n)$-periodic if every $i, j \in \mathbb{Z}$,
\begin{equation} \label{eqn:2.11}
\alpha_{i + mp ,j + nq} = \alpha_{i, j}.
\end{equation}
for all $p, q \in \mathbb{Z}$.

Let $\mathcal{P}_{B}(m,n)$ be the set of all
$(m,n)$-periodic patterns and $B$-admissible patterns. Let $\Gamma_{B}(m,n)=|\mathcal{P}_{B}(m,n)|$ be the number of all $(m,n)$-periodic and $B$-admissible patterns.

As in \cite{3}, $\mathcal{P}_{B}(m,n)$ can be expressed by trace
operators as follows.
%
%
%
%
%

From (\ref{eqn:2.9}), the periodic patterns in $\mathbf{Y}_{m+1}$ are given by $\mathbf{Y}_{m+1;i}$, $i = 1, 5, 9$. Define
\begin{equation} \label{eqn:2.12}
\mathbf{T}_m \equiv \sum_{i = 1, 5, 9} \mathbf{V}_{m+1; i}.
\end{equation}
$\mathbf{T}_{m}$ is called the trace operator of order $m$, as in \cite{3}.
Therefore, the following result is obtained.

\begin{pro}
\label{proposition:2.1}
Given $B\subseteq\Sigma_{2\times 2}(3)$, for $m,n\geq1$,
\begin{equation} \label{eqn:2.13}
\Gamma_B (m,n) = tr(\mathbf{T}_m^n ).
\end{equation}
\end{pro}

\begin{proof}
The proof is similar to that for corner coloring in \cite{2}. The details of the proof are omitted.
\end{proof}

Notably, from Proposition \ref{proposition:2.1}, $\mathcal{P}(B) \neq \emptyset$ if and only if
$
\Gamma_{B} (m,n)> 0
$
for some $m,n\geq 1$.

Recall some notation and terms from matrix theory. A matrix $\mathbf{A}$ is called nilpotent if $\mathbf{A}^k = 0$ for some $k \geq 1$.
The property "nilpotent" can be used to specify whether $B$ is a cycle generator or non-cycle generator.

\begin{pro}
\label{proposition:2.2}
Given a basic set $B \subset \Sigma_{2 \times 2}(3)$,
\begin{itemize}
    \item[(i)]  $B$ is a cycle generator if and only if $\mathbf{T}_m$ is not nilpotent for some $m \geq 1$.
    \item[(ii)] $\Sigma(B)=\emptyset$ if and only if $\mathbf{V}_m$ is nilpotent for some $m \geq 1$.
\end{itemize}
\end{pro}

\begin{proof}
From (\ref{eqn:2.13}) of Propositon \ref{proposition:2.1}, $B$ is easily seen to be a cycle generator if and only if $tr(\mathbf{T}_m^n)>0$ for some $m,n\geq1$. Therefore, (i) follows immediately.

Similarly, from (\ref{eqn:2.10}), (ii) follows.

\end{proof}

The following proposition provides an efficient method to check the nilpotent for non-negative matrix and can be easily proven. The proof is omitted.

\begin{pro}
\label{proposition:2.3}
Suppose $A$ is a non-negative matrix. Then, $\mathbf{A}$ is nilpotent if and only if $\mathbf{A}$ can be reduced to a zero matrix by repeating the following process:
if the $i$-th row (column) of $\mathbf{A}$ is a zero row, then the $i$-th colume (row) of $\mathbf{A}$ is replaced with a zero column.

\end{pro}
\section{ Main Result }

\subsection{Periodic Pairs}
\setcounter{equation}{0}
This section firstly classifies all local patterns in $\{ 0, 1, 2 \}^{\mathbb{Z}_{2\times 2}}$ into three groups.

First, the local pattern  $\alpha =
\begin{array}{c}
\psfrag{a}{\tiny{$\alpha_0$}}
\psfrag{b}{\tiny{$\alpha_2$}}
\psfrag{c}{\tiny{$\alpha_{1}$}}
\psfrag{d}{\tiny{$\alpha_{3}$}}
\includegraphics[scale=0.31]{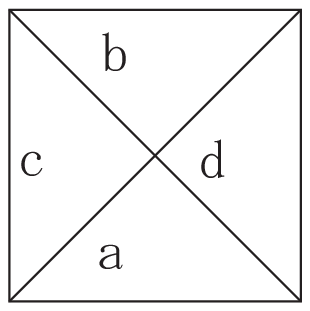}
\end{array}
=  \alpha_0, \alpha_1, \alpha_2, \alpha_3 ) $ is assigned a number by

\begin{equation} \label{eqn:3.1}
\varphi( ( \alpha_0, \alpha_1, \alpha_2, \alpha_3 ) ) = 1 + \sum_{j=0}^3 \alpha_j 3^j.
\end{equation}
Then, all 81 local patterns are listed in the following three groups $G_0$, $G_1$ and $G_2$.
\begin{eqnarray*}
G_0 &=& \{ 1, 11, 21, 31, 41, 51, 61, 71, 81 \} \\
G_1 &=& \{ 2, 3, 4, 7, 10, 12, 14, 17, 19, 20, 24, 27, 28, 32, 33, 34, 38, 40, \\
       & &42, 44, 48, 49, 50, 54, 55, 58, 62, 63, 65, 68, 70, 72, 75, 78, 79, 80 \} \\
G_2 &=& \{ 5, 6, 8, 9, 13, 15, 16, 18, 22, 23, 25, 26, 29, 30, 35, 36, 37, 39, \\
       & &43, 45, 46, 47, 52, 53, 56, 57, 59, 60, 64, 66, 67, 69, 73, 74, 76, 77 \}
\end{eqnarray*}

\begin{equation*}
\begin{array}{c}
G_{0} =
\begin{array}{c}
\includegraphics[scale=0.4]{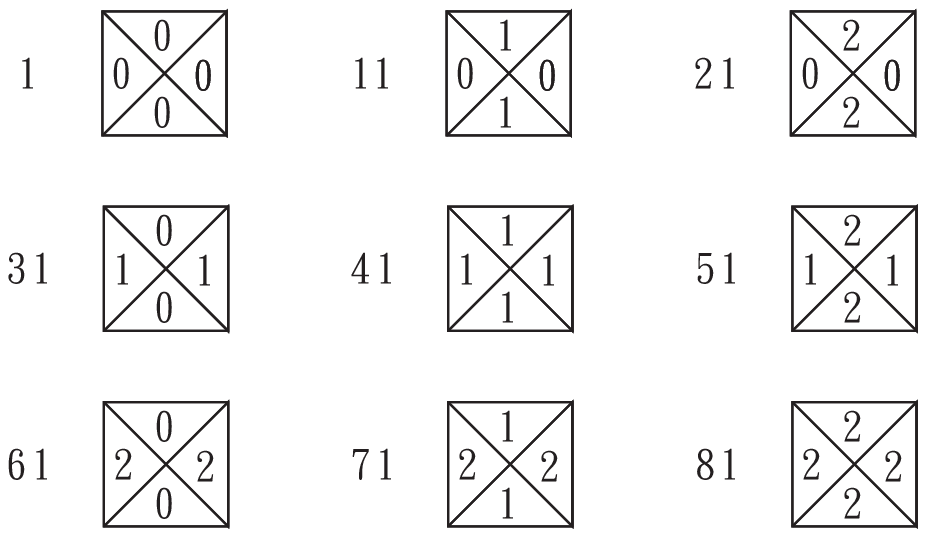}
\end{array} \\ \\ \\
G_{1} =
\begin{array}{c}
\includegraphics[scale=0.4]{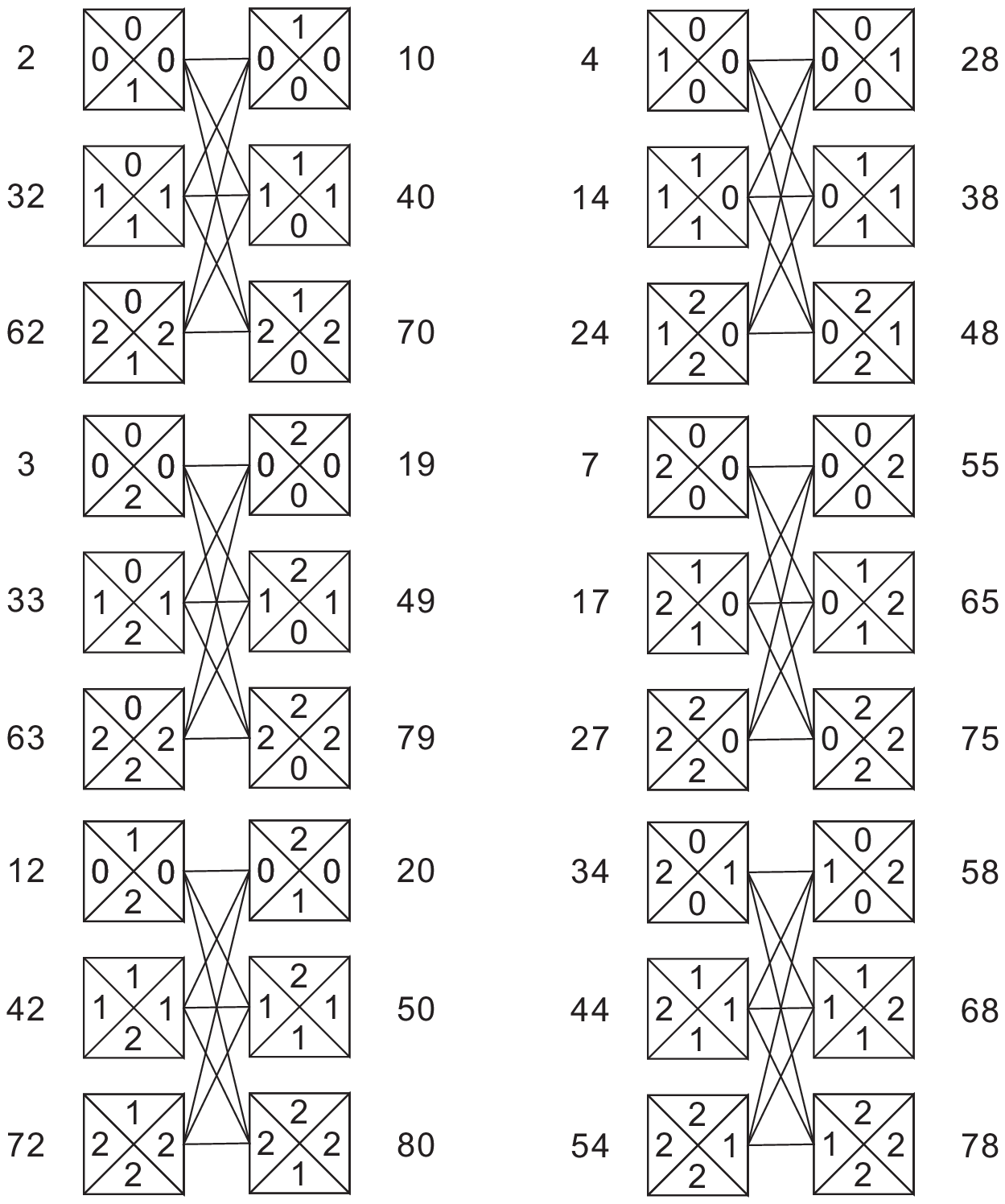}
\end{array}
\\ \\ \\
G_{2} =
\begin{array}{c}
\includegraphics[scale=0.4]{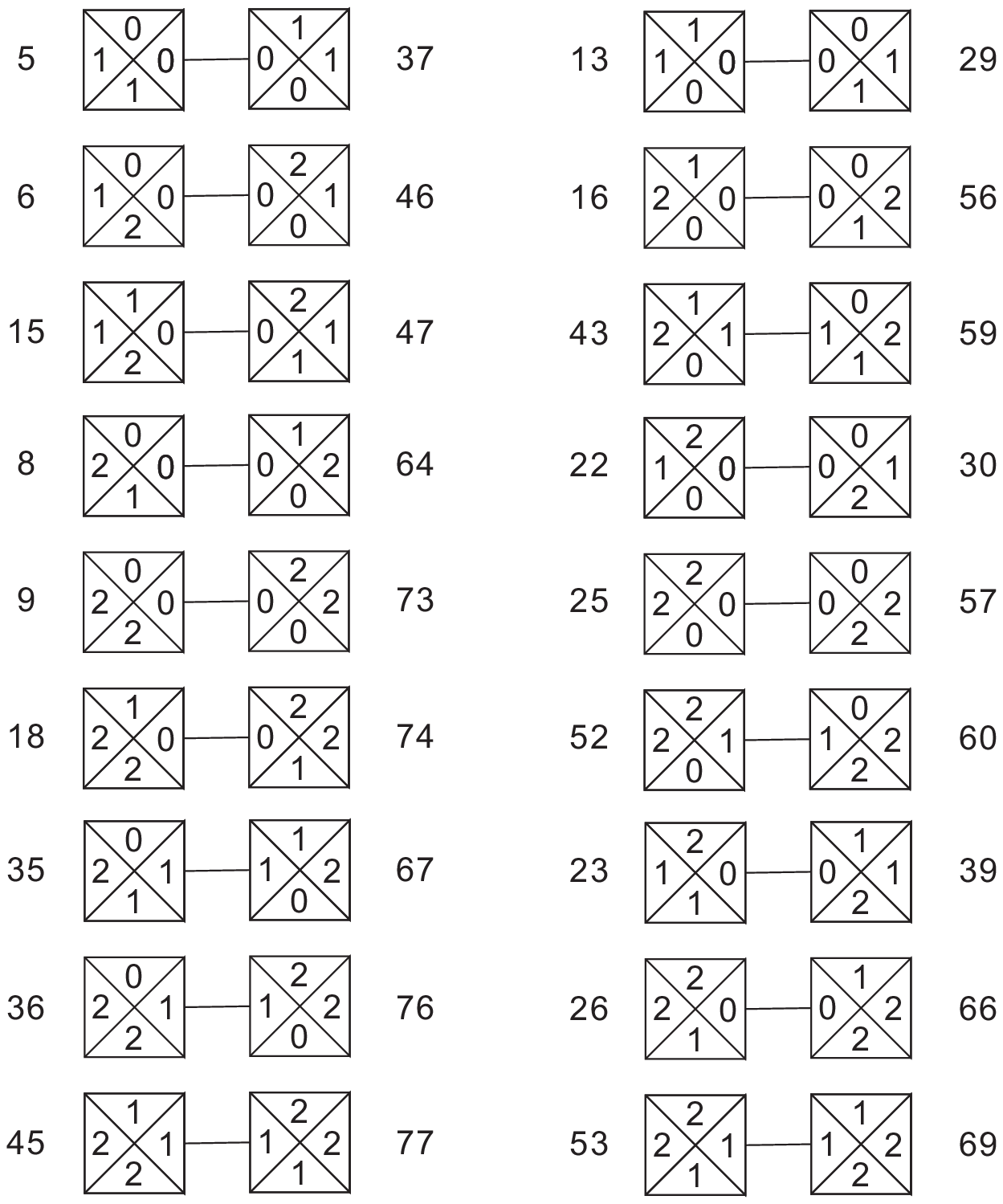}
\end{array}
\end{array}
\end{equation*}
\begin{equation*}
\text{Figure 3.1}
\end{equation*}
Clearly, every tile in $G_0$ can generate a $(1,1)$-periodic pattern. Furthermore, elements in $G_i$, $i=1,2$, form periodic pairs as in Fig. 3.1: two tiles that are connected by a line can generate a $(2,2)$-periodic pattern.
More precisely, the diagrams in Fig. 3.1 can be interpreted as follows.

\begin{pro}
\label{proposition:3.1}
\begin{itemize}
\item[(i)] Each tile $e$ in $G_0$ can generate a periodic pattern by repetition of itself: $\{ e \}$ is then a minimum
            cycle generator and is $(1,1)$-periodic.
\item[(ii)] For each tile $e$ in $G_1$, there exist exactly three tiles $e_1,e_2,e_3\in G_{1}$ such that $\{ e, e_i \}$ can form a periodic cycle, which is a $(2,2)$-periodic pattern, and $\{ e, e_i \}$ is a minimum cycle generator for $1 \leq i \leq 3$.
\item[(iii)]  For each tile $e$ in $G_2$, there exists exactly one tile $e^{\prime}\in G_{2}$ such that $\{ e, e^{\prime} \}$ is a minimum
            cycle generator and is $(2,2)$-periodic.
\end{itemize}
\end{pro}
These minimum cycle generators are the simplest.

 \begin{rmk}
 \label{remark:3.2}
From Fig. 3.1, the set $G_1 \cup G_2$ with 72 tiles can be decomposed into 36 disjoint sets that each consists of two tiles such that each set is a minimal cycle generator. Therefore, the number of the tiles of maximal non-cycle generators in $\mathcal{N}(3)$ is equal to or less than 36. Indeed,  elements in each pairs can be carefully picked up, and the maximum non-cycle generator with 36 elements thus obtained; see Table A.3. Moreover, from Proposition \ref{proposition:2.2}, for each element $B$ in these eight equivalence classes, $\Sigma(B)=\emptyset$ can be verified.
 \end{rmk}

%
%
%
%
%
%
%
%
%
%
%

\subsection{Algorithms }

Before the developed algorithms are presented , some notation must be introduced.
\begin{defi}
\begin{itemize}
\item[(i)] For a set $A$, let $\mathbb{P}(A)$ be the power set of $A$.

\item[(ii)] For $\mathbb{B}\subseteq \mathbb{P}(\Sigma_{2\times 2}(3))$, let

\begin{equation*}
[\mathbb{B}]=\left\{ [B] \hspace{0.1cm}\mid \hspace{0.1cm} B\in \mathbb{B} \right\}.
\end{equation*}

\item[(iii)] For $[B]\in [\mathbb{P}(\Sigma_{2\times2}(3))]$, let $\langle B\rangle$ be a fixed chosen element of $[B]$.

\item[(iv)] Let $\mathcal{N}^{*}(3)$ be the set of all maximal non-cycle generators that can not generate a global pattern. Indeed,
\begin{equation*}
\mathcal{N}^{*}(3)=\left\{ N\in \mathcal{N}(3) \hspace{0.1cm} \mid \hspace{0.1cm} \Sigma(N)=\emptyset\right\}.
\end{equation*}

\item[(v)] For $B\subseteq\Sigma_{2\times 2}(3)$, let $\mathcal{C}(B)$ be the set of all minimal cycle generators that are contained in $B$.

\item[(vi)] For $B\subseteq\Sigma_{2\times 2}(3)$, let $\mathcal{N}^{*}(B)$ be the set of all maximal non-cycle generators that can not generate a global pattern and are contained in $B$.

\end{itemize}
\end{defi}

Now, the main idea of the algorithms is introduced, as follows.

\begin{equation*}
\begin{array}{l}
\text{Let } \\
\hspace{1.0cm }N=2^{81},             \\  \\
\hspace{1.0cm }\mathbb{P}(\Sigma_{2\times 2}(3))= \{ B_{j} \hspace{0.1cm}\mid \hspace{0.1cm} 0\leq j\leq N-1 \},\text{ where } B_0= \emptyset, \\ \\
\hspace{1.0cm }\text{Initial state for }\mathcal{C}(3)\text{ : }\mathcal{C}_{I}(0)=\{\emptyset\},\\ \\
\hspace{1.0cm }\text{Initial state for }\mathcal{N}^{*}(3)\text{ : } \mathcal{N}^{*}_{I}(0)=\{\emptyset\}, \\ \\
\hspace{1.0cm }\text{Initial state for the set of aperiodic sets}\text{ : }\mathcal{U}_{I}(0)=\{\emptyset\}. \\
\end{array}
\end{equation*}

\begin{equation*}
\begin{tabular}{lllllllllllllllllllll}
\hline
\textbf{Main Algorithm} & & & & & &  & &  & & & & &&&&& &\\
\hline
$j = 0$ \\
\textbf{repeat} \\
 \hspace{0.3cm}$j =  j+1 $ \\
 \hspace{0.3cm}\textbf{if} $\mathcal{P}(B_{j})\neq \emptyset$,\\
\hspace{0.7cm} $\mathcal{C}_{I}(j)=\mathcal{C}_{I}(j-1)\cup \{B_{j}\}$ \\
 \hspace{0.3cm}\textbf{else}\\
 \hspace{0.7cm} \textbf{if} $\Sigma(B_{j})= \emptyset$, \\
\hspace{1.1cm} $\mathcal{N}^{*}_{I}(j)=\mathcal{N}^{*}_{I}(j-1)\cup \{B_{j}\}$ \\
 \hspace{0.7cm}\textbf{else}\\
 \hspace{1.1cm} $\mathcal{U}_{I}(j)=\mathcal{U}_{I}(j-1)\cup \{B_{j}\}$ \\
  \hspace{0.7cm}\textbf{end}\\
  \hspace{0.3cm}\textbf{end}\\
\textbf{until} $j=N-1$ \\
\hline
\end{tabular}
\end{equation*}

%
%
After the algorithm has been executed, if $\mathcal{U}_{I}(N-1)=\{\emptyset\}$, then Wang's conjecture holds for $p=3$.
The methods to achieve the goal are introduced below.

\begin{itemize}
\item[(I)] reduce the number of cases that must be considered in the computation,

\item[(II)] construct efficient initial states for $\mathcal{C}(3)$ and $\mathcal{N}^{*}(3)$,

\item[(III)]  construct an efficient process for determining whether or not $\mathcal{P}(B_{j})=\emptyset$ and $\Sigma(B_{j})= \emptyset$.
\end{itemize}

With respect to (I), the decomposition $\Sigma_{2\times2}(3)=G_0\cup G_1 \cup G_2$ is used to reduce the number of cases that must be considered in the computation. Clearly, if $B\subseteq\Sigma_{2\times 2}(3)$ contains a tile $e\in G_0$, then $B$ is a cycle generator. Now, in studying Wang's conjecture, only cases $B\subseteq G_1 \cup G_2$ have to be considered.

Given $B=A_1 \cup A_2$ with $A_1 \in \mathbb{P}(G_1)$ and $A_2 \in \mathbb{P}(G_2)$, if $A_1$ or $A_2$ is a cycle generator, then $B$ immediately satisfies (\ref{eqn:1.2}). By (\ref{eqn:2.2}), the cases $B\subseteq G_1 \cup G_2$ that have to be considered can be further reduced to the cases in $\mathcal{I}$ or $\mathcal{I}'$:
\begin{equation} \label{eqn:3.2}
\begin{array}{rl}
\mathcal{I}\equiv & \left\{A_1 \cup \langle A_2\rangle \hspace{0.1cm} \mid \hspace{0.1cm} A_1 \in\mathcal{D}_{1} \text{ and  }[A_2]\in [\mathcal{D}_{2}] \right\} \\
& \\
=& \{B_{j} \hspace{0.1cm} \mid \hspace{0.1cm} 1\leq j \leq |\mathcal{I}| \}
\end{array}
 \end{equation}
and
 \begin{equation} \label{eqn:3.3}
 \mathcal{I}' \equiv\left\{\langle A_1\rangle \cup A_2 \hspace{0.1cm} \mid \hspace{0.1cm} [A_1] \in [\mathcal{D}_{1}] \text{ and }A_2\in \mathcal{D}_{2} \right\},
  \end{equation}
where

\begin{equation}  \label{eqn:3.3-1}
\mathcal{D}_{j}=\left\{ A\in \mathbb{P}(G_{j}) \hspace{0.1cm} \mid \hspace{0.1cm} A \nsupseteq C \text{ for any }C\in \mathcal{C}(G_{j}) \right\}
\end{equation}
for $j=1,2$. For brevity, the proof is omitted. From Table 3.1, $N'\equiv|\mathcal{I}|\approx 1.35075\times 10^{12}$ and $|\mathcal{I'}|\approx 1.38458\times 10^{12}$.
Therefore, $\mathcal{I}$ is the better choice for reducing $B\subseteq G_1 \cup G_2$. Notably, $N'\ll |\mathbb{P}(G_1 \cup G_2 )|=2^{72}\approx 4.72237 \times 10^{21}$; the reduction is considerable. Table A.1 presents the details.

With respect to (II), let $\mathcal{U}_{I}(0)=\{\emptyset\}$. The initial data for $\mathcal{C}(3)$ are given by the set $\mathcal{C}_{I}(0)$ of all minimal cycle generators that are the subsets of $G_0$, $G_1$, or $G_2$. Indeed,
\begin{equation} \label{eqn:3.4}
\mathcal{C}_{I}(0)=\underset{j=0}{\overset{2}{\cup}}\mathcal{C}(G_j).
\end{equation}
On the other hand, the initial data for $\mathcal{N}^{*}(3)$ are given by
\begin{equation} \label{eqn:3.5}
\mathcal{N}_{I}^{*}(0)= \left\{N\in\mathcal{N}^{*}(G_1 \cup G_2):|N|=36\right\}.
\end{equation}
From Remark 3.2, $\mathcal{N}_{I}^{*}(0)$ equals the set of all maximal non-cycle generators in $G_1 \cup G_2$ with 36 tiles. $\mathcal{C}_{I}(0)$ and $\mathcal{N}^{*}_{I}(0)$ can be easily found using a computer program. See Table A.2 and A.3.

With respect to (III), the flowchart, which is based on (I) and (II), is as follows.
\begin{equation*}
\psfrag{a}{\tiny{Initial data: $\mathcal{C}_{I}(j-1)$, $\mathcal{N}_{I}^{*}(j-1)$ and $\mathcal{U}_{I}(j-1)$}}
\psfrag{b}{\tiny{Cosider each $B_{j}\in\mathcal{I}$, $1\leq j\leq N'$}}
\psfrag{c}{\tiny{$G_1$, $G2$ and $G3$ }}
\psfrag{d}{\tiny{
Find $[\mathbb{P}(G_2)]^{-1}$ of $\mathbb{P}(G_2)$
}}
\psfrag{e}{\tiny{}}
\psfrag{f}{\tiny{Find $\mathcal{C}=\underset{j=0}{\overset{2}{\cup}}\mathcal{C}(G_j)$}}
\psfrag{g}{\tiny{Find $\mathcal{N}=\left\{N\in\mathcal{N}(G_1 \cup G_2):\right. $}}
\psfrag{h}{\tiny{$\left.|N|=36\right\} $}}
\psfrag{j}{\tiny{Consider each $B=A_1\cup A_2$, where}}
\psfrag{k}{\tiny{$A_1\in\mathbb{P}(G_1)$ and $A_2\in[\mathbb{P}(G_2)]^{-1}$}}
\psfrag{l}{\tiny{Check whether or not $B_{j}$ contains}}
\psfrag{m}{\tiny{an element $C\in\mathcal{C}_{I}(j-1)$}}
\psfrag{n}{\tiny{$\mathcal{P}(B_{j})\neq\emptyset$. $\left\{\begin{array}{l}\mathcal{C}_{I}(j)=\mathcal{C}_{I}(j-1) \\
\mathcal{N}^{*}_{I}(j)=\mathcal{N}^{*}_{I}(j-1) \\
\mathcal{U}_{I}(j)=\mathcal{U}_{I}(j-1)
\end{array}
\right.$}}
\psfrag{p}{\tiny{Check whether or not $B_{j}$ is a subset}}
\psfrag{q}{\tiny{of some $N\in\mathcal{N}_{I}^{*}(j-1)$}}
\psfrag{r}{\tiny{$\Sigma(B)=\emptyset$. $\left\{\begin{array}{l}\mathcal{C}_{I}(j)=\mathcal{C}_{I}(j-1) \\
\mathcal{N}^{*}_{I}(j)=\mathcal{N}^{*}_{I}(j-1) \\
\mathcal{U}_{I}(j)=\mathcal{U}_{I}(j-1)
\end{array}
\right.$}}
\psfrag{t}{\tiny{Check whether or not $\mathbf{V}_{k}(B_{j})$ is}}
\psfrag{u}{\tiny{ nilpotent for some $k\geq 1$}}
\psfrag{v}{\tiny{$\Sigma(B_{j})=\emptyset$. $\left\{\begin{array}{l}\mathcal{C}_{I}(j)=\mathcal{C}_{I}(j-1) \\
\mathcal{N}_{I}^{*}(j)=\mathcal{N}_{I}^{*}(j-1)\cup [B_{j}] \\
\mathcal{U}_{I}(j)=\mathcal{U}_{I}(j-1)
\end{array}
\right.$ }}
\psfrag{x}{\tiny{Check whether or not $\mathbf{T}_{k}(B_{j})$ is}}
\psfrag{y}{\tiny{ nilpotent for all $k\geq 1$}}
\psfrag{w}{\tiny{$\mathcal{P}(B_{j})\neq\emptyset$. $\left\{\begin{array}{l}\mathcal{C}_{I}(j)=\mathcal{C}_{I}(j-1)\cup [B_{j}] \\
\mathcal{N}_{I}^{*}(j)=\mathcal{N}_{I}^{*}(j-1) \\
\mathcal{U}_{I}(j)=\mathcal{U}_{I}(j-1)
\end{array}
\right.$ }}
\psfrag{z}{\tiny{$\Sigma(B_j)\neq\emptyset$ and $\mathcal{P}(B_j)=\emptyset$. }}
\psfrag{d}{\tiny{$\left\{\begin{array}{l}\mathcal{C}_{I}(j)=\mathcal{C}_{I}(j-1) \\
\mathcal{N}^{*}_{I}(j)=\mathcal{N}^{*}_{I}(j-1) \\
\mathcal{U}_{I}(j)=\mathcal{U}_{I}(j-1)\cup [B_{j}]
\end{array}
\right.$}}
\psfrag{0}{\tiny{Yes}}
\psfrag{1}{\tiny{No}}
\includegraphics[scale=0.6]{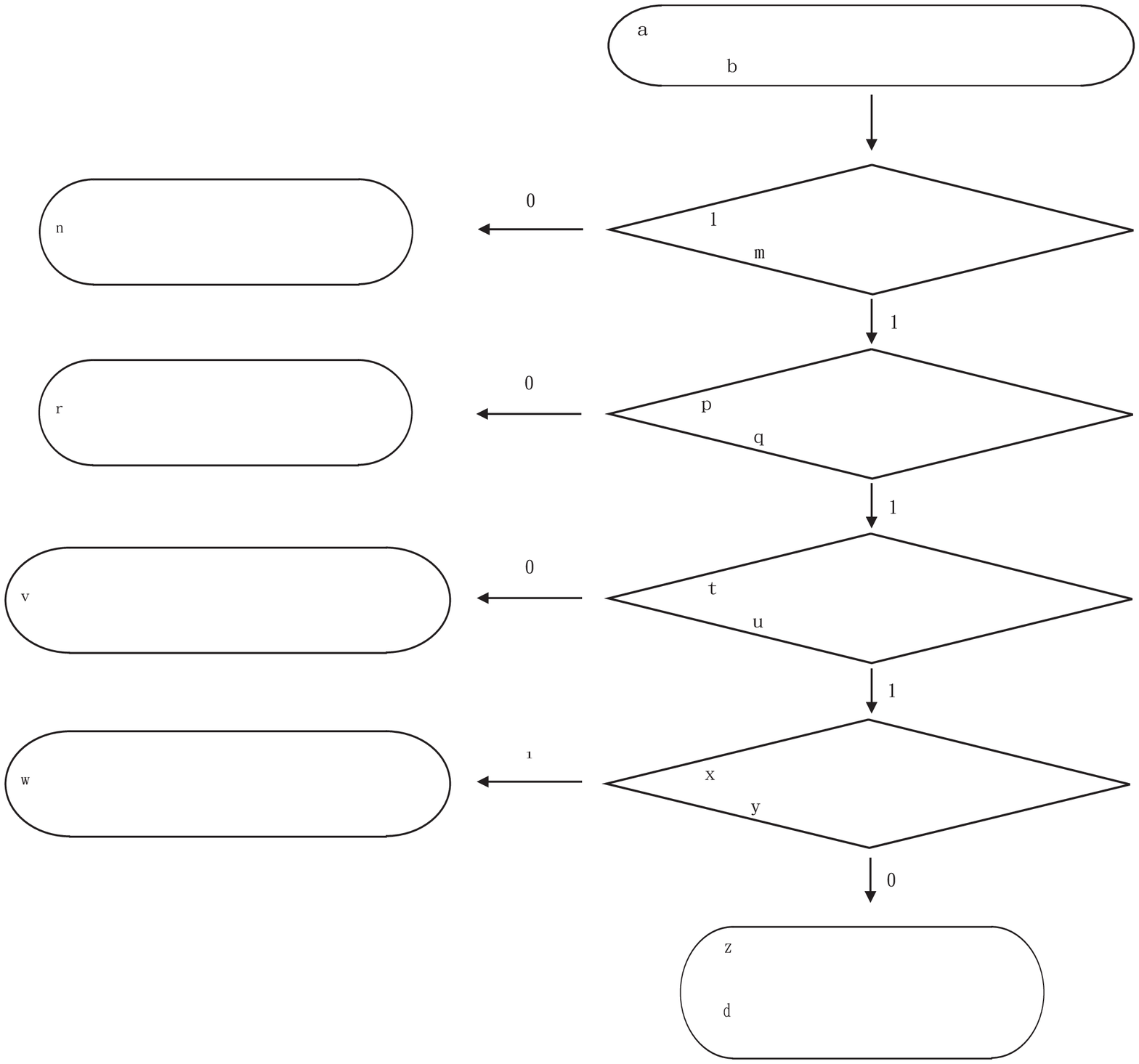}
\end{equation*}
\begin{equation*}
\text{Figure 3.2}
\end{equation*}

\begin{rmk}
\label{remark:3.4}
Suppose that the computation based on the flowchart has been completed. Let $\mathcal{C}=\mathcal{C}_{I}(N'-1)$, $\mathcal{N}^{*}=\mathcal{N}^{*}_{I}(N'-1)$ and $\mathcal{U}=\mathcal{U}_{I}(N'-1)$.
\begin{itemize}
\item[(i)] If the set $\mathcal{U}=\{\emptyset\}$, then Wang's conjecture holds for $p=3$; otherwise, every element in $\mathcal{U}$ is an aperiodic set.

\item[(ii)] It is easy to see that an element in $\mathcal{C}$ may be not a minimal cycle generator. However, $\mathcal{C}(3)$ can be obtained from $\tilde{\mathcal{C}}\equiv\underset{C\in\mathcal{C}}{\bigcup}[C]$ by the following process. If $C_1,C_2\in\tilde{\mathcal{C}}$ with $C_1\subsetneq C_2$, then $C_2$ must be removed from $\tilde{\mathcal{C}}$. Indeed,

\begin{equation*}
\mathcal{C}(3)=\left\{C\in \tilde{\mathcal{C}}\hspace{0.1cm} \mid \hspace{0.1cm} C\text{ does not contain any }C'\in \tilde{\mathcal{C}}\text{ except itself} \right\}.
\end{equation*}

\item[(iii)]  In a manner similar to that for (ii), let $\tilde{\mathcal{N}}\equiv\underset{N\in\mathcal{N}^{*}}{\bigcup}[N]$. Now,

 \begin{equation*}
\mathcal{N}^{*}(3)=\left\{N\in \tilde{\mathcal{N}}\hspace{0.1cm} \mid \hspace{0.1cm} N\text{ is not a proper subset of }N' \text{ for all }N'\in \tilde{\mathcal{N}}\text{ except itself}\right\}.
\end{equation*}
Moreover, if $\mathcal{U}=\{\emptyset\}$, $\mathcal{N}(3)=\mathcal{N}^{*}(3)$.
\end{itemize}
\end{rmk}%

\subsection{Main result }

The computer program of Fig. 3.2 is written, and the computation is completed in finite time. Indeed, the cases that consume the most time are those in which the numbers of tiles in $B\subseteq G_{1}\cup G_{2}$ are 18 and 19.  These cases can be computed completely within a week. The main result is as follows.

\begin{thm}
\label{theorem:3.5}
The set $\mathcal{U}$ is equal to $\{\emptyset\}$, and Wang's conjecture holds for $p=3$.
\end{thm}

\begin{rmk}
\label{remark:3.6}
\begin{itemize}
\item[(i)] $\mathcal{C}(3)$ and $\mathcal{N}(3)$ can be obtained and their numbers are listed in Table A.4.

\item[(ii)] The computational results reveal that the maximum orders $m$ of $\mathbf{T}_{m}(B)$ and $\mathbf{V}_{m}(B)$ in applying Proposition \ref{proposition:2.2} are $m=35$ and $m=13$, respectively. More precisely, for $B'=\{2,5,13,36,53,60,62,64,77\}$, $\mathbf{T}_{35}(B')$ is not nilpotent but $\mathbf{T}_{k}(B')$ is nilpotent for $1\leq k\leq 34$. On the other hand, for $B''=\{2,4,5,6,9,13,14,16,18,27,32,39,60,67,78,79\}$, $\mathbf{V}_{13}(B'')$ is nilpotent and $\mathbf{V}_{k}(B'')$ are not nilpotent for any $1\leq k\leq12$. The analytic proof that these numbers are maximal is not available. A prior estimate of the upper bound of $m$ does not exist.
\end{itemize}

\end{rmk}

For completeness, Tables A.4 and A.5 give the numbers of minimal cycle generators and maximal non-cycle generators.

\section*{Appendices }

\section*{A.1}
Table A.1 presents the numbers of $\mathcal{D}_{j}$ and $[\mathcal{D}_{j}]$, $j=1,2$. Denote by
\begin{equation*}
g_j (k)=\text{ the number of }\mathcal{D}_{j} \text{ with }k \text{ tiles}
\end{equation*}
and
\begin{equation*}
\bar{g}_j (k)=\text{ the number of }[\mathcal{D}_{j}] \text{ with }k \text{ tiles}
\end{equation*}
for $1\leq j\leq 2$ and $1\leq k\leq18$.

\newpage

\begin{equation*}
\begin{tabular}{|l|c|c|c|c|}
\hline
$k$ &  $g_1 (k)$ &  $\bar{g}_1 (k)$  & $g_2 (k)$ & $\bar{g}_2 (k)$  \\ \hline
1 & 36 & 1 & 36 & 1 \\ \hline
2 & 576 & 8& 612& 8\\ \hline
3 & 5304& 31& 6504& 34\\ \hline
4 & 31032& 146& 47988& 219\\ \hline
5 & 122184& 475& 256320& 971\\ \hline
6 & 342204& 1290& 998136& 3692\\ \hline
7 & 711288& 2581& 2812752& 10043\\ \hline
8 & 1129896& 4092& 5771988& 20554\\ \hline
9 & 1397892& 5005& 8886612& 31338\\ \hline
10 & 1361448& 4903& 10558368& 37319\\ \hline
11 & 1047816& 3763& 9807336& 34539\\ \hline
12 & 635580& 2321& 7125612& 25253\\ \hline
13 & 300888& 1106& 4007484& 14203\\ \hline
14 &109080 & 423& 1708632& 6162\\ \hline
15 & 29304& 118& 533664& 1945\\ \hline
16 &5508 & 28& 115164& 453\\ \hline
17 & 648& 4& 15336& 65\\ \hline
18 & 36& 1& 948& 8\\
\hline
\end{tabular}%
\end{equation*}

\begin{equation*}
\text{Table A.1}
\end{equation*}

\section*{A.2}
Table A.2 presents the equivalence classes of the minimal cycle generators in $G_{1}$ and $G_{2}$.
\begin{equation*}
\begin{tabular}{|c|c|}
\hline
$k$&   $[C]\in[\mathcal{C}(G_1)]$ with $k$ tiles  \\
\hline
$2$ & $[\{  2 , 10  \}]$ \\
\hline
& $[\{   2 , 40  \}]$ \\
\hline
$3$ & $[\{  2 , 12 ,19  \}]$ \\
\hline
 & $[\{  2 , 12,49  \}]$ \\
\hline
 & $[\{  2 , 42,79  \}]$ \\
\hline
\end{tabular}%
\end{equation*}
\begin{equation*}
\text{Table A.2 (a) }
\end{equation*}

\newpage

\begin{equation*}
\begin{tabular}{|c|c|}
\hline
$k$ &   $[C]\in[\mathcal{C}(G_2)]$ with $k$ tiles  \\
\hline
$2$ & $[\{  5 , 37  \}]$ \\
\hline
 $3$ & $[\{  5,45,73  \}]$ \\
\hline
 $4$ & $[\{  5 , 13,  30  ,46 \}]$ \\
\hline
 & $[\{   5 , 16,  30 , 73
  \}]$ \\
\hline
 & $[\{    5  ,16 , 39 , 74
 \}]$ \\
\hline
 & $[\{    5,   9  ,46 , 64
 \}]$ \\
\hline
& $[\{   5  ,13  ,35 , 64
  \}]$ \\
\hline
$5$ & $[\{    5  ,13 , 30 , 43 , 74\}]$ \\
\hline
 & $[\{    5  ,15 , 30 , 52 , 73
  \}]$ \\
\hline
 & $[\{    5 ,  9  ,39 , 46 , 76
  \}]$ \\
\hline
 & $[\{    5 ,  6,  35,  52  ,64
  \}]$ \\
\hline
& $[\{    5 , 15,  35 , 46 , 66
  \}]$ \\
\hline
$6$ & $[\{    5 , 13 , 35,  43  ,57 , 73
  \}]$ \\
\hline
 & $[\{    5  , 6  ,43 , 53 , 66  ,73
  \}]$ \\
\hline
& $[\{    5 , 13 , 36  ,53 , 66 , 73
 \}]$ \\
\hline
 & $[\{   5 ,  9  ,39 , 52  ,67  ,74
  \}]$ \\
\hline
 & $[\{   5  , 9  ,39 , 43 ,74 , 76
  \}]$ \\
\hline
& $[\{    5 , 13  ,30  ,43  ,47,  64
 \}]$ \\
\hline
& $[\{   5  , 6 , 13,  43  ,47  ,66
  \}]$ \\
\hline
& $[\{    5  , 9  ,13  ,25  ,39  ,74
 \}]$ \\
\hline
 & $[\{   5  , 9  ,13 , 30,  47 , 64
  \}]$ \\
\hline
 & $[\{   5  , 6  ,16  ,47  ,57  ,64
  \}]$ \\
\hline
& $[\{    5 ,  9 , 16,  53 , 66 , 74
  \}]$ \\
\hline
 & $[\{    5 ,  9  ,13 , 39,  53,  74
  \}]$ \\
\hline
& $[\{   5 ,  6 , 13 , 30,  52 , 73
  \}]$ \\
\hline
 & $[\{   5 ,  9,  15 , 43,  60,  74
  \}]$ \\
\hline
 & $[\{    5  ,13 , 35 , 45 , 66 , 74
  \}]$ \\
 \hline
$7$ & $[\{   5  , 9 , 13 , 30 , 52 , 64  ,74
  \}]$ \\
\hline
 & $[\{   5  , 9 , 13  ,47 , 52 , 57 , 64
  \}]$ \\
\hline
& $[\{    5  , 6 , 16 , 36 , 53,  66 , 73
  \}]$ \\
\hline
 & $[\{   5  , 9  ,16 , 39,  47 , 69 , 76
  \}]$ \\
\hline
 & $[\{   5  , 6 , 13 , 35 , 43,  66,  73
 \}]$ \\
\hline
& $[\{   5  , 6 , 16,  35,  39 , 47  ,76
  \}]$ \\
\hline
& $[\{    5 ,  9  ,13  ,30  ,52,  56  ,64
 \}]$ \\
\hline
& $[\{   5 ,  6 , 16 , 35 , 36 , 57  ,73
 \}]$ \\
\hline
 & $[\{    5  , 9 , 15 , 22 , 46,  56 , 66
 \}]$ \\
\hline
 & $[\{    5  ,15 , 25  ,35  ,45  ,64  ,74
  \}]$ \\
\hline
$8$ & $[\{    5  , 9  ,13  ,26 , 35  ,43  ,57 , 74
  \}]$ \\
\hline
 & $[\{    5   ,9 , 13  ,35  ,39 , 52 , 74  ,76
  \}]$ \\
\hline
& $[\{    5,   9  ,13 , 45 , 47  ,52 , 56 , 64
  \}]$ \\
\hline
\end{tabular}%
\end{equation*}
\begin{equation*}
\text{Table A.2 (b) }
\end{equation*}

\section*{A.3}
Table A.3 shows the equivalence classes of maximal non-cycle generators with $36$ tiles.

\begin{equation*}
\footnotesize{
\begin{tabular}{|c|}
\hline
1.\hspace{0.2cm} $\begin{array}{ll}[\{  2 , 3,  4,  5 , 6  ,7 , 8 , 9 ,12 ,13 ,14, 15, 16, 17 ,18 ,22 ,23 ,24 , 25 ,\\ \hspace{0.2cm}26 ,27 , 32 , 33, 34 ,35 ,36 ,42 ,43 ,44 ,45 ,52 ,53, 54, 62 ,63, 72
\}]\end{array}$   \\
\hline
2.\hspace{0.2cm} $\begin{array}{ll}[\{  2 , 3,  4,  5 , 6 , 7 , 8  ,9 ,12, 13 ,14, 15 ,16, 17, 18, 22 ,23, 24, 25, \\ \hspace{0.2cm}26, 27, 32 ,33 ,34 ,35 ,36 ,42, 43, 44, 45 ,53, 54, 60, 62 ,63 ,72
\}]\end{array}$   \\
\hline
3.\hspace{0.2cm} $\begin{array}{ll}[\{  2 , 3  ,4 , 5,  6,  7 , 8 , 9, 12 ,13 ,14 ,15 ,16, 17 ,18, 22 ,23 ,24, 25, \\ \hspace{0.2cm}26, 27, 32 ,33, 34 ,35, 36 ,42 ,43 ,44 ,45, 54, 60, 62 ,63 ,69 ,72
\}]\end{array}$   \\
\hline
4.\hspace{0.2cm} $\begin{array}{ll}[\{  2  ,3  ,4 , 5,  6 , 7 , 8 , 9, 12 ,13 ,14 ,15, 16 ,17 ,18, 22, 23 ,24 ,25 ,\\ \hspace{0.2cm}26 ,27, 32 ,33 ,34, 35, 36, 42 ,44 ,45, 54 ,59, 60 ,62 ,63 ,69 ,72
\}]\end{array}$   \\
\hline
5.\hspace{0.2cm} $\begin{array}{ll}[\{  2,  3 , 4 , 5,  6,  7 , 8  ,9 ,12, 13, 14 ,15, 16, 17 ,18 ,22, 23 ,24 ,26, \\ \hspace{0.2cm}27 ,32, 33, 35 ,36, 42, 45 ,57, 58, 59 ,60, 62, 63 ,68 ,69, 72 ,78
\}]\end{array}$   \\
\hline
6.\hspace{0.2cm} $\begin{array}{ll}[\{  2 , 3,  4,  5,  6,  7,  8,  9 ,12 ,13, 14, 15, 16 ,17 ,18 ,22 ,23, 24, 26, \\ \hspace{0.2cm}27, 32 ,33 ,35 ,36 ,42 ,57, 58 ,59 ,60 ,62 ,63 ,68 ,69, 72 ,77, 78
\}]\end{array}$   \\
\hline
7.\hspace{0.2cm} $\begin{array}{ll}[\{  2 , 3,  4,  5,  6,  7 , 8  ,9 ,12 ,13 ,14 ,15, 16 ,17 ,18 ,22 ,23, 24, 26 ,\\ \hspace{0.2cm}27 ,32 ,33 ,36 ,42 ,57, 58 ,59, 60 ,62, 63 ,67 ,68, 69 ,72, 77 ,78
\}]\end{array}$   \\
\hline
8.\hspace{0.2cm} $\begin{array}{ll}[\{  2  ,3  ,4 , 5 , 6 , 7 , 8 , 9, 12, 13 ,14, 15 ,16, 17 ,18, 22 ,23 ,24, 27,\\ \hspace{0.2cm} 32 ,33 ,36, 42, 45, 57, 58 ,59, 60 ,62, 63 ,66 ,67, 68 ,69, 72 ,78
\}]\end{array}$   \\
\hline
\end{tabular}%
}
\end{equation*}

\begin{equation*}
\text{Table A.3}
\end{equation*}

\section*{A.4}
Table A.4 shows the numbers of $\mathcal{C}(3)$ and $\mathcal{N}(3)$. Firstly, denote by
\begin{equation*}
\left\{
\begin{array}{l}
\mathcal{C}_{3}(k)= \left\{B\in \mathcal{C}(3) \hspace{0.1cm}: \hspace{0.1cm}|B|=k \right\}, \\
\mathcal{N}_{3}(k)= \left\{N\in \mathcal{N}(3) \hspace{0.1cm}:\hspace{0.1cm}|N|=k \right\},\\
\mathcal{C}_{3,e}(k)= \left\{[B]\in [\mathcal{C}(3)] \hspace{0.1cm}:\hspace{0.1cm} |B'|=k \text{ for all }B'\in [B] \right\}, \\
\mathcal{N}_{3,e}(k)= \left\{[N]\in [\mathcal{N}(3)] \hspace{0.1cm}: \hspace{0.1cm} |N'|=k \text{ for all }N'\in [N] \right\}.
\end{array}
\right.
\end{equation*}
Clearly, from Proposition 3.1, $\mathcal{C}(3)=\underset{k=1}{\overset{36}{\bigcup}}\mathcal{C}_{3}(k)$ and $\mathcal{N}(3)=\underset{k=1}{\overset{36}{\bigcup}}\mathcal{N}_{3}(k)$. Only the cases for $\mathcal{C}_{3}(k)\neq\emptyset$ and $\mathcal{N}_{3}(k)\neq\emptyset$ are listed.

\begin{equation*}
\begin{tabular}{|l|c|c|}
\hline
$k$ &   $|\mathcal{C}_{3}(k)|$ & $|\mathcal{C}_{3,e}(k)|$  \\
\hline
10  & 2880 & 10  \\
\hline
9 & 84600& 301  \\
\hline
8& 305388 & 1094 \\
\hline
7  & 264384 & 952\\
\hline
6 & 105012& 406  \\
\hline
5  & 21060& 102 \\
\hline
4 & 3672 & 29 \\
\hline
3  & 528 & 8\\
\hline
2 & 72 & 3 \\
\hline
1 & 9 & 1 \\
\hline
\end{tabular}%
\end{equation*}
\begin{equation*}
\text{Table
A.4.(a)}
\end{equation*}

\begin{equation*}
\begin{tabular}{|l|c|c|}
\hline
$k$ &  $|\mathcal{N}_{3}(k)|$& $|\mathcal{N}_{3,e}(k)|$  \\
\hline
36  & 1296 & 8 \\
\hline
34  & 720 & 3\\
\hline
32& 1152  & 4 \\
\hline
31  & 3168 & 11\\
\hline
30  & 576 & 2\\
\hline
29  & 288 & 1\\
\hline
28  & 3168 & 12\\
\hline
27  & 3456& 12 \\
\hline
26 & 6048 & 21 \\
\hline
25 & 5760 & 20 \\
\hline
24 & 5184& 18  \\
\hline
23 & 6624& 23  \\
\hline
22 & 8640 & 30 \\
\hline
21 & 12672& 44  \\
\hline
20 & 20160 & 70 \\
\hline
19  & 35280& 123 \\
\hline
18& 50256  & 175 \\
\hline
17  & 90000 & 313\\
\hline
16 & 93024 & 324 \\
\hline
15  & 108720& 379 \\
\hline
14  & 120384 & 422\\
\hline
13  & 148536& 522 \\
\hline
12  & 163512 & 576\\
\hline
11 & 157536& 556  \\
\hline
10  & 186480& 657 \\
\hline
9 & 133200 & 483 \\
\hline
8  & 42624 & 156\\
\hline
7  & 2160& 9 \\
\hline
\end{tabular}%
\end{equation*}
\begin{equation*}
\text{Table A.4 (b)}
\end{equation*}

\section*{Acknowledgments}
The authors want to thank Prof. Wen-Wei Lin for suggesting the use of the concept of nilpotence to identify cycle and non-cycle generators.

\end{document}